\documentclass[a4paper,11pt]{article}
\usepackage[utf8]{inputenc}
\usepackage[hmargin=2.7cm, vmargin=3.5cm,marginparwidth=2cm]{geometry}
\usepackage{bbm}
\usepackage{caption}

\RequirePackage{amsthm,amsmath,amsfonts,amssymb,amstext,mathtools}
\RequirePackage[numbers]{natbib}
\RequirePackage[colorlinks,citecolor=blue,urlcolor=blue]{hyperref}
\RequirePackage{graphicx}
\usepackage{bbm}
\usepackage{enumitem}
\usepackage[dvipsnames]{xcolor}
\usepackage{hyperref}
\usepackage{tikz}
\usetikzlibrary{arrows,spy}
\usetikzlibrary{patterns}

\definecolor[named]{linkcolor}{RGB}{0,26,153}
\hypersetup{
  colorlinks,
  linkcolor=linkcolor,
  citecolor=linkcolor,
  urlcolor=linkcolor,
}

\numberwithin{equation}{section}
\theoremstyle{plain}
\newtheorem{theorem}{Theorem}[section]

\newtheorem{lemma}[theorem]{Lemma}
\newtheorem{proposition}[theorem]{Proposition}

\theoremstyle{remark}
\newtheorem{remark}[theorem]{Remark}

\theoremstyle{definition}

\newtheorem{definition}[theorem]{Definition}

\newcommand{\N}{\mathbb{N}}
\newcommand{\Z}{\mathbb{Z}}

\newcommand{\R}{\mathbb{R}}
\renewcommand{\P}{\mathbb{P}}
\newcommand{\E}{\mathbb{E}}
\renewcommand{\L}{\mathbb{L}}

\newcommand{\ind}{\mathbbm{1}}
\newcommand{\bblock}{\mathtt{block}}
\newcommand{\ttime}{\mathrm{time}}
\newcommand{\dens}{\mathrm{dens}}
\newcommand{\reff}{\mathrm{ref}}
\newcommand{\FP}{\mathrm{FP}}
\newcommand{\couple}{\mathrm{couple}}
\newcommand{\spread}{\mathrm{spread}}
\newcommand{\Pois}{\mathrm{Pois}}

\newcommand{\conf}{\mathrm{conf}}
\newcommand{\maximal}{\mathrm{max}}
\newcommand{\mart}{\mathrm{mart}}

\DeclareMathOperator{\Supp}{Supp}
\DeclareMathOperator{\Var}{Var}
\DeclarePairedDelimiter{\norm}{\lVert}{\rVert}
\DeclarePairedDelimiter{\nnorm}{\big\lVert}{\big\rVert}
\DeclarePairedDelimiter{\Norm}{\Big\lVert}{\Big\rVert}
\DeclareMathOperator*{\argsup}{arg\,sup}
\DeclareMathOperator*{\arginf}{arg\,inf}

\begin{document}

\title{Ancestral lineages for a branching annihilating random walk}

\author{Pascal Oswald}
\date{}

\maketitle

\begin{abstract}
  We study ancestral lineages of individuals of a stationary discrete-time branching annihilating random walk (BARW) on the $d$-dimensional lattice $\Z^d$. Each individual produces a Poissonian number of offspring with mean $\mu$ which then jump independently to a uniformly chosen site with a fixed distance $R$ of their parent. By interpreting the ancestral lineage of such an individual as a random walk in a dynamical random environment, we obtain a law of large numbers and a functional central limit theorem for the ancestral lineage.
\end{abstract}



\section{Model and main result}

We reconsider a model of discrete-time branching annihilating random walk (BARW) on $\Z^d$, $d\ge 1$, that was first examined in \cite{BARW}. There, model-specific parameter regimes were identified for which the BARW survives with positive probability and for which it exhibits a unique non-trivial ergodic equilibrium. Building upon these results, we are interested in the long-term statistical properties of the position of a single individual's ancestors in a population evolving as a stationary BARW. Understanding such properties for locally self-regulating models, of which the BARW is an example, is not only compelling from a purely mathematical point of view. In the context of mathematical ecology, for instance, it is natural to translate questions on the spatial distribution of ``types'' in a population into questions on the spatial embedding of genealogies of sampled individuals, cf.\ \cite{MR2759587}. See also \cite{etheridge2023looking} for recent results in this direction (among others) regarding a broad class of models of locally regulated populations.

Our main results are that  ancestral lineages satisfy a law of large numbers as well as an annealed central limit theorem, cf.\ Theorem~\ref{thm:main} below.

The main tool in proving these results is a renormalisation construction, together with a result from \cite{RWancestry2016BirknerCernyDepperschmidt}. In \cite{RWancestry2016BirknerCernyDepperschmidt} Birkner, \v{C}ern\'{y} and Depperschmidt develop an abstract programme,
which allows to study ancestral lineages of spatial population models with local self-interactions lying in the universality class of oriented percolation, via a renewal argument. More precisely, the authors work out conditions that imply a law of large numbers and an annealed central limit theorem for the spatial paths of ancestral lineages for individuals drawn from a stationary population. We check that the BARW satisfies these conditions by exploiting and adapting the renormalisation construction that was used to show the existence of a unique stationary law in \cite{BARW}.

Let us now introduce the model. We study a discrete-time Markov process $\eta$ with state space $\{0,1\}^{\Z^d}$ and denote by $\eta_n(z)$ the state
of site $z \in \Z^d$ at time $n \in \N$ (later, when dealing with the stationary process, we take $n \in \Z$). We interpret $\eta_n(z)=1$ as the site $z$ being occupied
by a single particle at time $n$ and $\eta_n(z)=0$ as the site being vacant. In order to describe the dynamics of $\eta$ we introduce the following notational conventions: we write $\norm{ \,\cdot\, }$ for the uniform norm on $\Z^d$ and let
${B_R(z) = \{ x \in \Z^d : \norm{ z-x} \le R \}}$ be the $d$-dimensional
ball (box) of radius $R \in \N$ centred at $z\in \Z^d$.
Moreover, we set $V_R:=2R+1$ to be its side length, so that the volume of $B_R(z)$ is $V_R^d$.

For fixed $R \in \N$, $\mu >0$, and an initial particle configuration
$\eta_0 \in \{0,1\}^{\Z^d}$, the configurations $\eta_n$ at times $n\ge 1$ are
obtained recursively through a three-step procedure. Let
$z \in \Z^d$ be such that $\eta_n(z)=1$. Then, in a first step, the particle at site $z$ dies and gives birth to a Poisson number of offspring with mean $\mu$. Secondly, each offspring moves independently to a uniformly chosen site in $B_R(z)$. Lastly, whenever there are two or more particles at a given site, then all the particles at that site are removed, i.e.\ annihilated.
The particles remaining after the annihilation step make up the configuration $\eta_{n+1}$.

The thinning and superposition properties of the Poisson distribution
give the following equivalent description of the model.
For any configuration $\eta \in \{ 0,1 \}^{\Z^d}$ and $z \in \Z^d$, define first
the (local) density of particles at $z$ by
\begin{equation}
  \label{eqn:density}
  \delta_R(z; \eta) := V_R^{-d} \sum_{x \in B_R(z)} \eta(x).
\end{equation}
Then, in order to get from a configuration $\eta_n$ to $\eta_{n+1}$ we fix $\eta_n \in \{0,1\}^{\Z^d}$ and denote by $N_{n+1}(z)$ the number of newborn particles at $z$ in generation $n+1$ after the dispersal step but before annihilation has occurred. This number is given by the superposition of the offspring of all particles that can move to $z$ (that is, of all $x \in B_R(z)$ with $\eta_n(x)=1$). Thus, using the notation of \eqref{eqn:density}, $N_{n+1}(z)$ is a Poisson random variable with mean $\mu \delta_R(z; \eta_n )$. Taking the annihilation into account, it then holds that
\begin{equation}
  \label{eqn:next_gen_barw_0}
 \eta_{n+1}(z) =
 \begin{cases} 1 \quad \text{if } N_{n+1}(z) = 1, \\
 0 \quad \text{otherwise.}
 \end{cases}
 \end{equation}
Let
\begin{equation}
  \label{eqn:varphi}
  \varphi_{\mu}(w) := \mu w\, e^{-\mu w}, \quad w \in [0,\infty)
\end{equation}
denote the probability that a Poisson random variable with mean $\mu w$
equals 1. By construction, the random variables in the family
$(\eta_{n+1}(z): z \in \Z^d)$ are conditionally independent given $\eta_n$
and by \eqref{eqn:next_gen_barw_0}, \eqref{eqn:varphi} we can represent
our system as
\begin{equation}
  \label{eqn:next_gen_barw}
  \eta_{n+1}(z) =
  \begin{cases} 1 \quad \text{with probability }
    \varphi_{\mu}( \delta_R(z; \eta_n)), \\
    0 \quad \text{otherwise}.
  \end{cases}
\end{equation}
This gives a representation of $\eta$ as a particular example of a probabilistic cellular automaton (PCA), see e.g.\ \cite{mairesse2014around} for an introduction to PCA. Such PCA can be seen as discrete-time counterparts to interacting particle systems, in which the entire system updates ``in parallel'', as opposed to ``sequentually'' as is true for interacting particle systems. 

\begin{remark}
\begin{enumerate}[label = (\alph*)]
    \item  The representation \eqref{eqn:next_gen_barw} relies on the offspring distribution being Poissonian and is not possible for a non-Poissonian choice of offspring distributions. Moreover, the constructions used subsequently depend very delicately on properties of $\varphi_\mu$.  We refer to \cite[Section~1.2]{BARW} for a detailed discussion of the model assumptions and possible generalisations. 
    \item The BARW is non-monotone/non-attractive in the sense of particle systems. That is, adding more particles to a given generation, does not guarantee an increase in the number of particles in the succeeding generation as a higher number of particles leads to more annihilation. A consequence of this seemingly simple fact is that many tools for monotone systems (e.g.\ comparisons using coupling, subadditivity arguments) are not directly applicable.
\end{enumerate}
\end{remark}
In \cite{BARW} the existence of (non-trivial) invariant measures for the BARW is examined (note that by \eqref{eqn:next_gen_barw} the \emph{empty configuration} $\mathbf{0}\in \{0,1\}^{\Z^d}$ is always an absorbing state). We summarise the relevant statements for the current objective in the following proposition.

\begin{proposition}[{Survival and complete convergence, \cite{BARW}}]\label{prop:survivalandcc} Let $\mu\in (1,e^2)$.
    \begin{enumerate}[label =(\roman*)]
        \item There exists $R_\mu$ such that for any $R>R_\mu$ the process $\eta$ survives with positive probability and has a unique non-trivial invariant extremal distribution $\nu_{\mu,R}$.
        \item Conditioned on non-extinction, the law of $\eta_n$ converges to $\nu_{\mu,R}$ in the weak topology.
    \end{enumerate}
\end{proposition}
We are only interested in the parameter regime for which there is a unique non-trivial extremal invariant distribution. Therefore, we only consider $\mu \in (1,e^2)$ and $R>R_\mu$ in the rest of the paper. By doing so, the existence of $\nu_{\mu,R}$ is always guaranteed by  Proposition~\ref{prop:survivalandcc}. 

Let us now introduce the main object of interest, namely the ancestral lineages of single particles in a BARW. We consider the stationary process $\eta = (\eta_n)_{n \in \Z}$ such that for each $n \in \Z$, $\eta_n$ is distributed as $\nu_{\mu,R}$. It is clear from the informal description of the model, cf.\ the paragraph before \eqref{eqn:density}, that the model can be enriched with genealogical information, by relating child and parent particles.

For simplicity of notation, we condition $\eta$ on having a particle at the space-time origin and always consider the ancestral lineage of the particle at the space-time origin.
Conditioned on $\{\eta_0(0) = 1\}$, the dynamics of the ancestral lineage of the particle at the origin are  described by the time-inhomogeneous Markov chain $X = (X_k)_{k\in \N_0}$ given by
\begin{equation}\label{def:RW}
    X_0 = 0,\quad \text{and}\quad \P(X_{k+1} = y | X_k = x,\eta) = p_{\eta}(k;x,y), \quad k\ge 1
\end{equation}
where the transition probabilities $p_\eta(k;x,y)$ are given by
\begin{equation}
\label{eqn:randomwalk}
    p_{\eta}(k;x,y) :=\frac{\eta_{-k-1}(y)}{\sum_{z\in B_R(x)}\eta_{-k-1}(z)}.
\end{equation}
Indeed, contingent on the random walk being at site $x$ at time $k$, for any $y \in B_R(x)$ the number of particles sent from $y$ to $x$ is $\Pois(\mu V_R^{-d})$ distributed, conditional on the total sum over all $y \in B_R(x)$ being equal to one. Now, since a vector of independent Poisson random variables, conditioned on the total size of its sum, has a multinomial distribution, it follows readily that a particle ``selects'' its predecessor uniformly among all particles alive one generation earlier  which are within distance $R$. That is, the transition kernel in \eqref{def:RW} can be written as in \eqref{eqn:randomwalk}.

The random walk $X$ defined in \eqref{def:RW} is a random walk in the (relatively complicated) random environment $\eta$ and describes the space-time embedding of ancestral lineages of particles in $\eta$. Hence, the randomness of $X$ comes solely from the genealogy of the particle. We stress, moreover, that the forwards in time direction of the random walk, corresponds to the backwards in time direction of the environment $\eta$. 

Our main result states that $X$ satisfies a law of large numbers and a central limit theorem when averaging over the genealogical randomness (i.e.\ randomness due to $X$ taking steps) and the randomness in the environment. To this end, we write $P_\eta$ for the conditional law of $\P$, given $\eta$, so that $p_\eta(k;x,y) = P_\eta(X_{k+1}=y | X_k = x)$. Moreover, we $\P_0(\cdot) := \P( \cdot | \eta_0(0) = 1)$ and write $\E_0$ and $E_\eta$ for the corresponding expectations.

\begin{theorem}\label{thm:main}
For $\mu \in (1,e^2)$ there exists $\widetilde R_\mu\ge R_\mu$ such that for $R \ge \widetilde R_\mu$ the random walk as defined in \eqref{def:RW}--\eqref{eqn:randomwalk} satisfies
    \begin{equation}
        \label{eqn:lln}
        P_\eta \big(k^{-1}X_k \stackrel{ k\to \infty}{\longrightarrow}0 \big) = 1 \quad \text{for}\quad \P_0(\cdot )-\text{a.a. } \eta.
    \end{equation}
    Moreover for any $g \in C_b(\R^d)$
    \begin{equation}
        \label{eqn:clt}
    \E_0\big[g(k^{-1/2}X_k)\big] \stackrel{ k\to \infty}{\longrightarrow} \E[g(Z)], 
    \end{equation}
    where $Z$ is a (non-degenerate) centred isotropic $d$-dimensional normal random variable, i.e.\ $Z \sim \mathcal{N}(0,\sigma^2I)$ for some $\sigma^2>0$, where $I$ is the identity matrix.  Moreover, a functional version of \eqref{eqn:clt} holds as well.
\end{theorem}

The proof of Theorem~\ref{thm:main} amounts to verifying the applicability of \cite[Theorem~3.1]{RWancestry2016BirknerCernyDepperschmidt}, which gives abstract conditions for a law of large numbers and a central limit theorem for random walks in dynamic random environments to hold. 

\subsubsection*{Related literature}
The questions addressed in \cite{RWancestry2016BirknerCernyDepperschmidt} for a general class of population models with local self-interactions were motivated by earlier work of the same authors. In \cite{Birkner2013backbone} random walks on the backbone of an oriented percolation cluster on $\Z^d,d\ge1$ (which correspond to ancestral lineages of a discrete-time contact process) are considered. A quenched law of large numbers and an annealed central limit theorem similar to Theorem~\ref{thm:main} above are shown. Moreover, the natural question of the behaviour of two random walks with transition probabilities as in \eqref{def:RW}, which corresponds to jointly describing the ancestral lineages of two distinct individuals, is considered. By controlling two copies of the random walk, the annealed central limit theorem was strengthened to a quenched one.
It follows moreover from \cite{Birkner2013backbone} that analogous annealed and quenched central limit theorems hold if one allows the discrete-time contact process to have random i.i.d.\ carrying capacities (i.e.\ every space-time site can carry a random number of particles, instead of just one). This was generalised in \cite{Miller2017weighted} to the case where the carrying capacity is not i.i.d.\ but mixing. It is shown that a quenched law of large numbers and an annealed central limit theorem hold under $\phi$-mixing in time 
(for $\phi_n \in O(n^{-1-\delta})$ resp.\ $\phi_n \in O(n^{-2-\delta})$ for some $\delta>0$) and a quenched central limit theorem under an exponential mixing in space and time, see \cite{Miller2017weighted,miller2017dissertation} for details. For unit carrying capacity, the quenched central limit theorem was recently extended to a quenched \emph{local} limit theorem for $d\ge 3$ in \cite{Bethuelsen2023local} (note that the dimensional constraint seems rectifiable as is commented upon in \cite[Outlook and open questions]{Bethuelsen2023local} ).

Moreover, the abstract programme developed in \cite{RWancestry2016BirknerCernyDepperschmidt} was applied to derive a quenched law of large numbers and an annealed central limit theorem for the  logistic branching random walk first studied in \cite{birkner2007survival} as a population model with logistic local self-regulation.
A comprehensive overview of these and related models can be found in \cite{Birkner2021ancestrallinages}.

\subsubsection*{Organisation of the article}
The rest of the article is organised as follows. In Section~\ref{sec:ancestry}, we formalise and make precise the conditions that need to be checked in order to apply the abstract machinery of \cite{RWancestry2016BirknerCernyDepperschmidt}, comment on why they are needed and prove Theorem~\ref{thm:main}. 
In Section~\ref{sec:renormalisation}, we introduce the main ideas behind the notion of goodness used for the renormalisation construction, which are based on suitable control of local densities of $\eta$. Lastly, Sections~\ref{sec:proof_main} and \ref{sec:proof_speed_prop} contain the proofs that the necessary conditions are indeed met.

\section{Abstract conditions and proof of Theorem~\ref{thm:main}}
\label{sec:ancestry}
Recall that our strategy to prove Theorem~\ref{thm:main} is to check the abstract conditions that let us apply Theorem~3.1 of \cite{RWancestry2016BirknerCernyDepperschmidt}.
These conditions are rather involved and not straightforward to present in isolation. We present them in the following as propositions which need to be proven. The proofs are given in Sections~\ref{sec:proof_main} and \ref{sec:proof_speed_prop}.

The conditions from \cite{RWancestry2016BirknerCernyDepperschmidt} can be divided into two parts. Firstly, into conditions on the random environment $\eta$, in which the random walk $X$ evolves, and secondly, into conditions on the random walk itself.

\subsection{Conditions on the environment}
\label{sec:assumptions_envrionment}
 The first condition on the environment (\cite[Assumption~3.2]{RWancestry2016BirknerCernyDepperschmidt}) is that it is Markovian and admits a ``local flow construction'' that allows to couple the process with different (and arbitrary) initial conditions.

To this end, we make use of an appropriate analogue of the \emph{graphical construction} of interacting particle systems, with which we can view the evolution of $\eta$ as a stochastic flow on its configuration space $\{0,1\}^{\Z^d}$. This offers the advantage of letting us define the process $\eta = (\eta_n)_{n\ge m}$ for all initial conditions $\eta_m \in \{0,1\}^{\Z^d}$, at any starting time $m\in \Z$, simultaneously. More precisely, 
we let $U(x,n)$, $x\in \Z^d$, $n\in \Z$, be a
collection of i.i.d.~uniform random variables on $[0,1]$. 
Then, for any $m \in \Z$ and any initial condition $\eta_m \in \{0,1\}^{\Z^d}$, 
we define, recursively for $n\ge 0$,
\begin{equation}
  \label{eqn:flow_construction}
  \eta_{m+n+1}(x) := \ind_{\{ U(x,m+n+1)
      \le \varphi_{\mu}(\delta_R(x;\eta_{m+n}))\}},
 \end{equation}
where $\delta_R(x;\eta_{m+n})$ is as in \eqref{eqn:density}. Comparing \eqref{eqn:flow_construction} to \eqref{eqn:next_gen_barw} with $m=0$ it follows immediately that the process defined by \eqref{eqn:flow_construction} has the law of the BARW. In this sense, the i.i.d.\ field $(U(x,n))_{(x,n) \in \Z^d \times \Z}$ of $\textrm{Unif}[0,1]$ random variables acts as \emph{driving noise} for the evolution of $\eta$. Moreover, the construction is local, as the value of $\eta$ at any space-time site $(x,n) \in \Z^d\times \Z$ is fully determined by the value $U(x,n)$ and by the values $\{\eta_{n-1}(y) : y \in B_R(x)\}$, and hence Assumption~3.2 of \cite{BARW} is satisfied.

Expanding on the idea of $\eta$ as a stochastic flow on the configuration space, we introduce for $-\infty<m<n$ the $\sigma$-algebras
\begin{equation}
    \label{eqn:sigma_algebra}
    \mathcal{G}_{m,n} := \sigma(U(x,k) : m< k\le n , x \in \Z^d).
\end{equation}
By iterating \eqref{eqn:flow_construction}, we can define a (random) family of $\mathcal{G}_{m,n}$-measurable mappings
\begin{equation}
    \label{eqn:stochastic_flow_1}
    \Phi_{m,n} : \{0,1\}^{\Z^d} \to \{0,1\}^{\Z^d},\quad -\infty<m<n,
\end{equation}
such that $\eta_n = \Phi_{m,n}(\eta_m)$. More precisely, we define for any $\zeta \in \{0,1\}^{\Z^d}$ and $x \in\Z^d$
\begin{equation*}
      \big(\Phi_{m,m+1}(\zeta)\big)(x) := \ind_{\{ U(x,m+1)
      \le \varphi_{\mu}(\delta_R(x;\zeta))\}}
\end{equation*}
and then set
\begin{equation*}
    \Phi_{m,n} := \Phi_{n-1,n}\circ \dots \circ \Phi_{m,m+1}.
\end{equation*}
By using these mappings, the dynamics of $(\eta_n)_{n \ge m}$ defined simultaneously for all initial conditions $\eta_m \in \{0,1\}^{\Z^d}$ and any $m \in \Z$.

The second condition on the environment (\cite[Assumption~3.3]{RWancestry2016BirknerCernyDepperschmidt}) asks that a comparison with supercritical oriented percolation can be made on an appropriately scaled space-time grid, using a very specific notion of ``good'' blocks.
This specific notion of goodness is introduced in Proposition~\ref{prop:main_proposition} below, where it is also stated that an appropriate coarse-graining exists for the BARW. 

To this end, we introduce the following notation. 
For spatial and temporal scales ${L_s,L_t \in \N}$ we consider space-time blocks whose ``bottom parts'' are centred at the points in the coarse-grained grid $\L :=L_s \Z^d\times L_t \Z$. Points in $\L$ are labeled by $\Z^d\times \Z$ such that $(x,n)\in\Z^d\times \Z$ is the label for the point $(L_sx,nL_t)\in\L$. 
For $m\in \N$ and $(x,n) \in \Z^d \times \Z$ we consider blocks
 \begin{equation}
    \label{eqn:new_blocks}
\bblock_m(x,n) := \big\{(y,k) \in \Z^d\times \Z : \norm{y-L_sx}\le m L_s , nL_t < k \le (n+1)L_t\big\}.
\end{equation}
Note that these blocks overlap in the spatial directions but never in the temporal direction. 
We further write for any $A \subseteq \Z^d\times \Z$, $U|_A$ for the restriction of the  field $U$ of driving noise to the set $A$. In particular, this means that for any $\bblock_m(x,n)$, the restriction of $U|_{\bblock_m(x,n)}$ is an element of $[0,1]^{B_{mL_s}(L_sx)\times \{1,\dots, L_t\}}$.

With this notation, we can present the second condition of \cite{RWancestry2016BirknerCernyDepperschmidt} on the environment in form of a proposition, the proof of which is given in Section~\ref{sec:proof_main} below.

\begin{proposition}
\label{prop:main_proposition}
    For any $\mu \in (1,e^2)$ and $\varepsilon>0$ there exists  $\widetilde R_{\mu,\varepsilon}\ge R_\mu$, such that for every $R\ge \widetilde  R_{\mu,\varepsilon}$ there is a spatial scale $L_s$, a temporal scale $L_t$, a set of good (local) configurations $G_\conf \subseteq \{0,1\}^{B_{2L_s}(0)}$ and a set of good (local) driving noise realisations $G_U \subseteq [0,1]^{B_{4L_s}(0)\times \{1,\dots,L_t\}}$ such that
    \begin{equation}
        \label{eqn:main_proposition}
    \P(U|_{\bblock_4(0,0)} \in G_U) \ge 1-\varepsilon
    \end{equation}
    and such that the following contraction and coupling conditions are satisfied: For any $(x,n) \in \Z^d\times \Z$ and any configurations $\eta_{nL_t}^{(1)},\eta_{nL_t}^{(2)} \in \{0,1\}^{\Z^d}$ at time $nL_t$, if $\eta_{nL_t}^{(i)}|_{B_{2L_s}(L_s x)}\in G_\conf$ for $i=1,2$ and $U|_{\bblock_4(x,n)} \in G_U$, then
    \begin{enumerate}[label =(\roman*)]
        \item \label{prop:item1} $\eta_{(n+1)L_t}^{(1)}(y)=\eta_{(n+1)L_t}^{(2)}(y)$ for all $\norm{y-L_sx}\le 3L_s$
        \item \label{prop:item2} $\eta_{(n+1)L_t}^{(1)}|_{B_{2L_s}(L_s(x+e))} \in G_\conf$ for all $e\in B_1(0)$.
    \end{enumerate}
    Moreover, when the $\eta_{nL_t}^{(i)}$'s agree on $B_{2L_s}(L_sx)$ (i.e.\ at the bottom centre of the block), then they agree on all space-time points in $B_{L_s}(L_s x) \times \{1,\dots, L_t\}$.
\end{proposition}

For a realisation $\eta$ of the BARW, we call a coarse-grained block $\bblock_4(x,n)$ \emph{good} if $U|_{\bblock_4(x,n)} \in G_U$ and $\eta_{nL_t}|_{B_{2L_s}(L_sx)} \in G_\conf$.
More formally, let for any $(x,n) \in \Z^d\times \Z$
\begin{equation}
    \label{eqn:oriented_percolation}
        \Gamma(x,n) := \ind_{\{ U|_{\bblock_4(x,n)} \in G_U\}}\ind_{\{\eta_{n L_t} |_{B_{2L_s}(L_s x)} \in G_\conf\}},
\end{equation}
then $\bblock_4(x,n)$ is \emph{good} if $\Gamma(x,n) = 1$. This notion of goodness can be viewed as a type of contractivity property of the local dynamics, in the sense that, on a good block, the flow $\Phi$, cf.\ \eqref{eqn:stochastic_flow_1}, tends to merge local configurations. 

\begin{remark}\label{rem:modblocks} 
    The block-construction from \cite{BARW}, which we recapitulate in Section~\ref{sec:blocks:cc}, does not satisfy Proposition~\ref{prop:main_proposition} verbatim. The issue is that, in addition to depending on the driving noise in a box, the notion of ``good block'' that was used there also depends on the \emph{specific} realisation of $\eta$ at the bottom of the block. 
   Nonetheless, we can reuse the ideas of \cite[Sectio~4]{BARW}, cf.\ Section~\ref{sec:proof_main} below, and this technicality can  be remedied by introducing (next to $G_U$) the set $G_\conf$ and by choosing slightly larger scales. The difference in the size of scales is elaborated upon in some more detail in Remark~\ref{rem:new_timescale} after having introduced some more details on the construction of \cite{BARW}.
\end{remark} 

\subsection{Conditions on the random walk}
\label{sec:assumptions_RW}
Let $X$ be the random walk as defined in \eqref{def:RW}, evolving in the dynamic random environment given by the stationary process $\eta$, which is defined as in \eqref{eqn:flow_construction} with parameters $\mu,R$ for which Proposition~\ref{prop:main_proposition} holds for some $\varepsilon>0$ and let $L_s,L_t \in \N$ be the scales corresponding to these parameters.

There are again two conditions in \cite{RWancestry2016BirknerCernyDepperschmidt}, on the random walk $X$, that need to be verified. The first is that if the random walk starts anywhere from the middle half of the top of a good block, i.e.\ a block $\bblock_4(x,n)$  such that $\Gamma(x,n) = 1$, for some $(x,n)\in \Z^d \times \Z$, then with high probability it doesn't cover long distances within the block, cf.\ \cite[Assumption~3.9]{RWancestry2016BirknerCernyDepperschmidt}.
The precise statement that needs to be checked is summarised in the following proposition. Recall for this that $P_\eta$ denotes the ``quenched'' probability measure, i.e.\ the measure, conditioned on a realisation of~$\eta$. 
\begin{proposition}
    \label{prop:speed_bound}
    For $\varepsilon,\delta>0$ there exists $R_{\mu,\delta,\varepsilon}>\widetilde R_{\mu,\varepsilon}$ such that for all $R\ge R_{\mu,\delta,\varepsilon}$ and $L_s, L_t, G_U, G_\conf$ as in Proposition~\ref{prop:main_proposition}, the random walk $X$ satisfies for $(x,n) \in \Z^d \times \Z$ 
    \begin{equation}
   \label{eqn:RW_stays_in_block}
   \min_{z: \norm{L_sx-z}\le L_s/2}P_\eta\Big(  \max_{(n-1)L_t<k\le nL_t}\norm{X_k-z}\le L_s/4 \Big| X_{(n-1)L_t}=z, \Gamma(x,n)=1 \Big)\ge 1-\delta.
\end{equation}
\end{proposition}
The proof of Proposition~\ref{prop:speed_bound} follows from the fact that in a block with good driving noise, the relative fluctuations of the local density of $\eta$ over $R$ balls and $r$-balls with $1\ll r<R$ are small. Thus, in each step, the increments of $X$ do not deviate much from the increments of a simple random walk. As the precise notion of goodness of a block depends on the construction in the proof of Proposition~\ref{prop:main_proposition},  and the specifics of the fluctuations of the density of $\eta$ are described in Section~\ref{sec:concentration} below, we postpone the details of the proof to Section~\ref{sec:proof_speed_prop}.

The second condition on $X$ is that it behaves symmetrically with respect to spatial point reflections, when $\eta$ is reflected accordingly, cf.~\cite[Assumption~3.11]{RWancestry2016BirknerCernyDepperschmidt}. As for any time $k\in \N$, the random walk in \eqref{def:RW} chooses uniformly among the possible ancestors of the particle at $X_k$, this condition holds trivially. 
Note that this symmetry corroborates that asymptotically the average speed of $X$ is zero.

With the results of Sections~\ref{sec:assumptions_envrionment} and \ref{sec:assumptions_RW} at hand, the proof of Theorem~\ref{thm:main} follows directly.
\begin{proof}[Proof of Theorem~\ref{thm:main}]
The assertion of the theorem follows by a combination of Proposition~\ref{prop:main_proposition} and \ref{prop:speed_bound} as well as Theorem~3.1 of \cite{RWancestry2016BirknerCernyDepperschmidt} for all $\eta$ defined with parameters $\mu \in (1,e^2)$ and for $R$ large enough.
\end{proof}

\section{Renormalisation construction}
\label{sec:renormalisation}

In this section, we first collect some results on the function $\varphi_\mu$ that will be essential in the proof of Proposition~\ref{prop:main_proposition} in Section~\ref{sec:proof_main}. These give insight into the behaviour of local densities $\delta_r(\cdot; \eta)$ for $1\ll r \le R$ which in turn lets us outline the idea of the block construction from \cite[Section~4]{BARW}, on which our construction in Section~\ref{sec:proof_main} is based. Moreover, we discuss how to get quantitative control of the local densities.

\subsection{Properties of \texorpdfstring{$\varphi_\mu$}{pm}}\label{ss:properties_of_varphi}
The following lemma summarises useful properties of the function  $\varphi_\mu(w) = \mu w e^{-\mu w}$, which appears in the definition of the dynamics of $\eta$, cf.\ \eqref{eqn:next_gen_barw} and \eqref{eqn:flow_construction}. 

\begin{lemma}\label{lem:varphi_properties}
\begin{enumerate}[label =(\alph*)]
  \item For $\mu >1$, $\varphi_\mu $ has two fixpoints,
  $0$ and $\theta_\mu := \mu^{-1} \log \mu$. The fixpoint $0$ is always repulsive.
  \item For $\mu \in (1, e^2)$, $\theta_\mu $ is an attractive fixpoint and for $\mu > e^2$, there are no attractive fixpoints.
  \item For every $\mu\in (1,e^2)$ there is $\varepsilon_\FP = \varepsilon_\FP(\mu)>0$ and
  $\kappa(\mu,\varepsilon_\FP)<1$ such that $\varphi_\mu $ is a contraction on
  $[\theta_{\mu}-\varepsilon_\FP, \theta_{\mu}+\varepsilon_\FP]$, that is,
  \begin{equation*}
    |\varphi_{\mu}(w_1)-\varphi_{\mu}(w_2)|
    \le \kappa(\mu,\varepsilon_\FP)|w_1-w_2|
    \qquad
    \text{for }w_1, w_2 \in
    [\theta_{\mu}-\varepsilon_\FP, \theta_{\mu}+\varepsilon_\FP].
  \end{equation*}
  \item There exists a strictly increasing sequence
  $\alpha_m \uparrow \theta_\mu$ and a strictly decreasing sequence
  $\beta_m \downarrow \theta_\mu$ such that
  $\varphi_{\mu} ([\alpha_m, \beta_m]) \subseteq(\alpha_{m+1}, \beta_{m+1})$
  for all $m \in \N$. Furthermore, it is possible to choose $\alpha_1 >0$
  arbitrarily small and $\beta_1 > 1/e$.
\end{enumerate}
\end{lemma}

\begin{proof}
    Properties (a)--(b) are a direct consequence of the definition of $\varphi_\mu$ and (c)--(d) are the contents of \cite[Lemma~4.1]{BARW}.
\end{proof}
Note that by Lemma~\ref{lem:varphi_properties}(d) for any $\mu \in (1,e^2)$ and any choice of $\alpha_1>0$ and $\beta_1>1/e$, there is a finite index $m_0$ such that $\alpha_m,\beta_m \in [\theta_\mu-\varepsilon_\FP,\theta_\mu+\varepsilon_\FP]$ for all $m\ge m_0$, i.e.
\begin{equation}
    \label{def:m_0}
        m_0 = m_0(\mu,\alpha_1,\beta_1) := \inf\big\{ m\ge 1 : \alpha_m,\beta_m \in [\theta_\mu-\varepsilon_\FP,\theta_\mu+\varepsilon_\FP]\big\}.
\end{equation}
This value will later play a role in establishing the properties Proposition~\ref{prop:main_proposition}\ref{prop:item1}--\ref{prop:item2}.

\subsection{Local densities and goodness}\label{sec:blocks:cc}
In  \cite[Section~4]{BARW} complete convergence of the BARW is shown by a comparison with supercritical oriented percolation with a coarse-graining and a notion of goodness on blocks that is reminiscent, but not identical to that of Proposition~\ref{prop:main_proposition} and \eqref{eqn:oriented_percolation}, cf.\ Remark~\ref{rem:new_timescale}. 
The exact notion of goodness in \cite{BARW} is tailored specifically to showing complete convergence, i.e.\ convergence (conditioned on survival) of the law of the BARW towards the \emph{unique} non-trivial extremal invariant distribution $\nu_{\mu,R}$. To achieve this, \emph{good} blocks were defined to make two distinct configurations which partially agree at the ``bottom'' of the block, evolve into configurations that agree on a larger portion of the ``top'' of the block (this should be compared to Proposition~\ref{prop:main_proposition}\ref{prop:item1}).

To motivate why it is reasonable to expect distinct local configurations of the BARW, following the same dynamics, to merge, recall that the random variables $(\eta_{n+1}(x))_{x\in \Z^d}$ are conditionally independent and conditionally Bernoulli distributed, given $\eta_n$, with respective parameters $\varphi_{\mu}(\delta_R(x;\eta_n))$, cf. \eqref{eqn:next_gen_barw}. 
Therefore, local densities $\delta_r(x;\eta_{n+1})$ (as sums of these conditional Bernoulli random variables) should concentrate around some value, at least for $r$ large. In \cite{BARW} only the case $r=R$ was considered and it was used there that the value around which the local densities concentrate converges to $\theta_\mu$, and thus, by repeated application of Lemma~\ref{lem:varphi_properties}(c)--(d) and \eqref{eqn:flow_construction}, the flow $\Phi$ tends to merge any two distinct realisations of the BARW over long enough time-spans. Controlling how distinct configurations merge under the dynamics of the BARW thus amounts to gaining control of local densities.

We follow the same idea as in \cite{BARW} to gain quantitative control of the (local) densities by introducing certain families of \emph{reference density profiles} $\zeta^{r,-}_k: \Z^d\to [0,1]$ and $\zeta^{r,+}_k: \Z^d\to [0,1]$, for $k\in \{0,\dots, k_0\}$. We are then interested in the situation where the local $r$-densities of the true system $\eta_k$ are wedged in between the two reference density profiles, i.e.\ when 
\begin{equation}\label{eqn:density_wedge}
    \zeta^{r,-}_k(x) \le \delta_r(x; \eta_k) \le \zeta^{r,+}_k(x),\quad \text{for suitable }\, x \in \Z^d, k \in \N_0 \,\text{ and }\,r\, \text{ large enough}.
\end{equation}

In Section~4.1 of \cite{BARW} a family of reference density profiles is introduced, which satisfies \eqref{eqn:density_wedge} with high probability for $r=R$. Moreover, the profiles are chosen in such a way that ${|\zeta^{R,+}_k- \zeta^{R,-}_k | \in [ \theta_\mu-\varepsilon_\FP, \theta_\mu + \varepsilon_\FP]}$ on part of the support of $\zeta_k^{R,-}$, and such that they have a fixed deterministic ``front'' that expands by a fixed distance in every time step, see Figure~\ref{fig:zeta_sketch} for a sketch of the one-dimensional profiles $\zeta_k^{R,\pm}$ of \cite{BARW}. The control \eqref{eqn:density_wedge} by the expanding families $\zeta^{R,\pm}$ thus lets one apply Lemma~\ref{lem:varphi_properties}(c)--(d) in a growing spatial region throughout a block. 
\begin{figure}[t]
  \label{fig:density_profile}
  \centering
  \includegraphics[width=\textwidth]{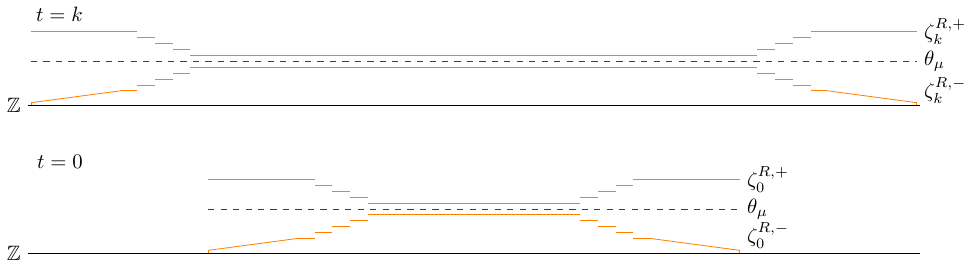}
  \caption{Sketch of the density profiles $\zeta_{k}^{R,\pm}$ (in orange) from \cite{BARW} for dimension $d=1$. The bottom shows the profiles $\zeta_0^{R,\pm}$, which then expand to $\zeta_k^{R,\pm}$ in the top image. The profiles are chosen such that the distance of the profiles to $\theta_\mu$ is smaller than $\varepsilon_\FP$ in the central constant part of the profiles.} 
  \label{fig:zeta_sketch}
\end{figure}

Based on this and a suitable coarse-graining of space-time, the notion of goodness in \cite{BARW} is defined using a two-step procedure for any two realisations $(\eta^{(1)}_n)_{n\ge 0}$ and $(\eta^{(2)}_n)_{n\ge0}$ of the BARW, with different initial conditions, coupled through the flow construction \eqref{eqn:flow_construction}. 
\begin{enumerate}[label =(\Roman*)]
    \item \label{item_A:old_well_startedness} A block based at some $( z,  t) \in \Z^d \times \Z$ is called \emph{well-started} if\begin{align}
    \label{eqn:well-started2}
    \delta_R(x; \eta_t^{(i)}) \in \big[ \zeta^{R,-}_0(x- z), \zeta^{R,+}_0(x- z) \big]
    \quad \text{for all } x \in  \{ z+y :\zeta^{R,-}_0(y)>0\}, \; i=1,2
  \end{align}
  and 
  \begin{align}
    \label{eqn:well-started2a}
    \eta_t^{(1)}(x) = \eta_t^{(2)}(x) \quad \text{for all $x$ in a suitably large ball around $z$.} 
  \end{align}
  \item \label{item_B:old_good} A block based at some $( z,  t) \in \Z^d \times \Z$ is called \emph{good} if it is well-started and the domination by the $\zeta^{R,\pm}$ profiles and the region where \eqref{eqn:well-started2a} holds spread throughout suitably large portions of the block (we refer to \cite{BARW} for details on what suitable means in this context).
\end{enumerate}

The first step is concerned with guaranteeing that the configurations of $(\eta^{(i)}_n)_{n\ge0}$ at the ``bottom'' of a block are controlled by a suitable shifted version of the reference profiles $\zeta_0^{R,\pm}$, whereas the second is concerned with the ``spreading'' of the region where the two configurations agree, as well as the region where the density control by the reference profiles holds. 

\begin{remark}
\label{rem:new_timescale}
\begin{enumerate}[label=(\alph*)]
    \item Morally, \eqref{eqn:oriented_percolation} is a ``uniformisation'' of the properties given by \ref{item_A:old_well_startedness}--\ref{item_B:old_good} above by making the notion of goodness independent of any specific configuration at the ``bottom'' of the block. To accomplish this, special care needs to be taken of \eqref{eqn:well-started2a}, which is a condition of two configurations partially agreeing in a region at the bottom of the block. To overcome this, we introduce below, in \eqref{eqn:reference_configurations} the set of reference configurations $\mathrm{C}_\reff$, given by all configurations $\widetilde \eta \in \{0,1\}^{\Z^d}$ such that the local densities are globally in the interval $[ \theta_\mu-\varepsilon_\FP,\theta_\mu + \varepsilon_\FP]$ to which configurations which are locally in $G_\conf$ will need to couple successfully.

    \item The size of blocks and thus the scales of the coarse-graining in the construction outlined by \ref{item_A:old_well_startedness}--\ref{item_B:old_good} are linked to concrete details of the reference density profiles $\zeta_k^{R,\pm}$ that are used, as the scales need to be chosen such that the desired properties spread to suitable large portions of the blocks.
\end{enumerate}
\end{remark}

In contrast to \cite[Section~4.1]{BARW}, for the current purposes, it does not suffice to only have control of the local $R$-densities of the process $\eta$, but in order to prove Proposition~\ref{prop:speed_bound}, we also need control of local $r_0$-densities for some $r_0<R$, which is specified in Section~\ref{sec:proof_speed_prop} as a fixed proportion of $R$. Without loss of generality, we assume that $r_0$ divides $R$. Importantly, taking $R$ large has the effect that $r_0$ is also large. To this end, we work with two sequences of reference density profiles in Sections~\ref{sec:proof_main}--\ref{sec:proof_speed_prop} below, one controlling the local $R$-densities and one controlling local $r_0$-densities, cf.\ \eqref{eqn:density_wedge}. 

\subsection{Concentration and comparison of local densities with deterministic profiles}\label{sec:concentration}

Let us now elaborate upon and formalise the idea of wedging local $r$-densities of $\eta_k$ in between two reference functions, as in \eqref{eqn:density_wedge}.
To this end, we introduce sequences of functions $\zeta_k^{r,\pm}$ on $\Z^d$, which serve as these reference functions and which we call \emph{comparison density profiles} (c.d.p.'s). These are analogous to the functions introduced in \cite[Definition~2.2]{BARW}, with the difference that we require an averaging property over $r$-balls instead of $R$-balls.

\begin{definition}
  \label{def:density_profiles}
  For given $\varepsilon, \delta  >0$ and $r \in \N $,  $(\varepsilon,\delta,r)$-\emph{comparison density profiles} ($(\varepsilon,\delta,r)$-\emph{c.d.p.}'s) are deterministic functions
  $\zeta_k^{r,-}, \zeta_k^{r,+} : \Z^d \to [0,\infty)$, $k=0,1,\dots,k_0$,
  satisfying:
  \begin{enumerate}[label = (\roman*)]
    \item \label{profile:i}
    For every $k=0,\dots,k_0$,
    $\zeta_k^{r,-}(\cdot) \le \zeta_k^{r,+}(\cdot)$.

    \item \label{profile:ii}
    For every $k=0,\dots,k_0$,
    $\Supp(\zeta_k^{r,-}):=\{x\in \Z^d: \zeta_k^{r,-}(x) > 0\}$ is finite, and
    $\zeta_k^{r,-}(x)\ge \varepsilon $ for
    every $x\in \Supp(\zeta_k^{r,-})$.

    \item \label{profile:iii}
    For every $k = 0,\dots,k_0-1$, and $x\in \Supp(\zeta_k^{r,-})$ it holds
    that if $\rho : B_{r}(x) \to \mathbb R$ satisfies $\rho (y) \in
    [\zeta_k^{r,-}(y), \zeta_k^{r,+}(y)]$ for all $y\in B_{r}(x)$, then
      \begin{equation}
      \label{eqn:xi.comp.bd1}
      (1+\delta) \zeta_{k+1}^{r,-}(x)
      \le V_{r}^{-d} \sum_{y \in B_{r}(x)} \varphi_\mu ( \rho(y))
      \le (1-\delta) \zeta_{k+1}^{r,+}(x).
    \end{equation}
  \end{enumerate}

\end{definition}

In Lemma~\ref{lem:zeta_is_cdp}, we show that such c.d.p.'s in fact exist, by constructing specific examples (in \eqref{eqn:new_zeta_minus}--\eqref{eqn:new_zeta_plus}). A comment is in order to clarify the role of the model-parameter $R$ in these c.d.p.'s and in particular in \eqref{eqn:xi.comp.bd1}.
Recall that by \eqref{eqn:next_gen_barw} the distribution of $\eta_{k+1}(x)$ is conditionally Bernoulli, given $\eta_k$, with parameter $\varphi_\mu(\delta_R(x;\eta_k))$. Therefore, in order to make any statement about $r$-densities at time $k+1$, we must have some knowledge of $R$-densities at time $k$. In terms of applying $(\varepsilon,\delta,r)$-c.d.p.'s to control local $r$-densities of $\eta$, we interpret $\rho$ in \ref{profile:iii} as representing values of local $R$-densities. This means that in order to make use of the control guaranteed by \eqref{eqn:xi.comp.bd1}, we need $R$-densities(!) at time $k$ to be controlled by the $(\varepsilon,\delta,r)$-c.d.p.'s. 

In the following, we introduce for any $\zeta^{r,\pm}$ satisfying Definition~\ref{def:density_profiles} with $r\in\{1,\dots,R\}$ and any $x\in \Z^d$ the sets of driving noise that provide uniform (in the configurations of $\eta$) $r$-density control at $x$ by the $(\varepsilon,\delta,r)$-c.d.p.'s centred at $x$ at time $k+1$ given that $R$-densities in $B_r(x)$ are controlled by $(\varepsilon,\delta,r)$-c.d.p.'s at time $k$.
\begin{equation*}
   \mathcal{U}_k^{r}(\zeta^{r,\pm};x):= \left\{ u \in [0,1]^{B_{r}(0)} :  \;
      \parbox{0.45\textwidth}{ for all configurations $\tilde\eta \in \{0,1\}^{\Z^d}$  s.t. $\delta_{R}(y; \tilde\eta)\in [\zeta_k^{r,-}(y), \zeta_k^{r,+}(y)]$ for all sites $y \in B_{r}(x)$, it holds that $V_{r}^{-d}\sum_{y \in B_{r}(x)}\ind_{\{u_{y-x} \le \varphi_\mu(\delta_{R}(y;\tilde\eta))\}}\in [\zeta_{k+1}^{r,-}(x),\zeta_{k+1}^{r,+}(x)]$} \: 
      \right\}.
\end{equation*}
Moreover, let $\mathcal{U}_k^{r,\pm}(\zeta^{r,\pm};x)$ be the corresponding sets, where only the upper/lower bound holds for the last sum, i.e.\  $\mathcal{U}_k^{r}(\zeta^{r,\pm};x) = \mathcal{U}_k^{r,-}(\zeta^{r,\pm};x)\cap \mathcal{U}_k^{r,+}(\zeta^{r,\pm};x)$. To ease the presentation, we drop the dependence on $\zeta^\pm$ in the notation, whenever the choice of the c.d.p.'s is clear and write simply $\mathcal{U}_k^{r}(x)$ and $\mathcal{U}_k^{r,\pm}(x)$, respectively.

\begin{lemma}\label{lem:U_k}
For $\varepsilon,\delta>0$, $r \in \N$, let $\zeta_k^{r,\pm}$ be a family of  $(\varepsilon,\delta,r)$-c.d.p.'s. It then holds for $k = 0,1,\dots,k_0-1$ that
    \begin{equation}
        \label{eqn:concentration_upgrade}
        \P\left(U|_{B_{r}(0)\times \{k\}} \in \mathcal{U}_k^r(0)\right) \ge 1-2e^{-cV_{r}^d},
    \end{equation}
    where $c = ({\delta \varepsilon})/({1/(2\delta \varepsilon) + 2/3})$.
\end{lemma}

\begin{proof}
    The proof is an adaptation of the proof of \cite[Lemma~2.3]{BARW}. Note that \eqref{eqn:concentration_upgrade} follows if one shows that  $\P\left(U|_{B_{r}(0)\times \{k\}} \in \mathcal{U}_k^{r,-}(0)\right)$ and $\P\left(U|_{B_{r}(0)\times \{k\}} \in \mathcal{U}_k^{r,+}(0)\right)$ are both greater than  $1-e^{-c V_{r}^d}$. We start by showing the first inequality.

    For $k\in\{0,\dots,k_0-1\}$ and $ y \in B_{r}(x)$ for $x \in \Z^d$, introduce the quantities
    \begin{equation*}
    \underline\lambda_k(y) := \arginf_{\lambda  \in [\zeta_k^{r,-}(y),\zeta_k^{r,+}(y)]}\varphi_\mu(\lambda). 
    \end{equation*}
    Note that by \eqref{eqn:xi.comp.bd1} it follows from $\underline\lambda_k(y) \in [\zeta_k^{r,-}(y),\zeta_k^{r,+}(y)]$ for all $y \in B_{r}(x)$, that 
    \begin{equation}
    \label{eqn:concentration_1}
        V_{r}^{-d}\sum_{y \in B_{r}(x)} \varphi_\mu(\underline\lambda_k(y))\ge (1+\delta)\zeta_{k+1}^{r,-}(x).
    \end{equation}
    Moreover, for $k \in \{0,\dots,k_0-1\}$ and $y \in B_{r}(x)$, let $Z_y^k:=\ind_{\{U(y,k) \le \varphi_\mu(\underline\lambda_k(y))\}}$ where $U$ is the i.i.d.\ field of driving noise. These are independent Bernoulli random variables with parameters $\varphi_\mu(\underline\lambda_k(y))$. 
    
    Using the $Z_y^k$ and \eqref{eqn:concentration_1}, we see that
    \begin{align*}
        \P\left(U|_{B_{r}(0)\times \{k\}} \notin \mathcal{U}_k^{r,-}(0)\right) &\le \P \Big( V_{r}^{-d}\sum_{y \in B_{r}(x)}\ind_{\{U(y,k) \le \varphi_\mu(\underline\lambda_k(y))\}}<\zeta_{k+1}^{r,-}(x) \Big)\\
        &\le \P\Big( \sum_{y \in B_{r}(x)}\big(Z_y^k-\varphi_\mu(\underline\lambda_k(y))\big)<-\delta V_{r}^d \zeta_{k+1}^{r,-}(x)\Big).
    \end{align*}
    Since the right hand side of the above is a sum of independent centred random variables, and since
    \begin{equation*}
    \Var\Big( V_{r}^{-d}\sum_{y \in B_{r}(x)}Z_y^k\Big)
 = V_{r}^{-2d}\sum_{y \in B_{r}(x)} \varphi_\mu(\underline\lambda_k(y))\big( 1-\varphi_\mu(\underline\lambda_k(y))\big) \le \frac{1}{4}V_{r}^{-d},
 \end{equation*}
a concentration estimate for the sum of independent Bernoulli random variables, see e.g.\ \cite[Lemma~A.1]{BARW} for an estimate based on Bernstein's concentration inequality, can be used, exactly as in the proof of \cite[Lemma~2.3]{BARW} in order to conclude that
 \begin{equation*}
 \P\left(U|_{B_{r}(0)\times \{k\}} \notin \mathcal{U}_k^{r,-}(0)\right) \le \exp(-cV_{r}^d),
 \end{equation*}
 where $c = ({\delta \varepsilon})/({1/(2\delta \varepsilon) + 2/3})$ is due to details of the concentration estimate.
 
 Proving that the probability of $\{U|_{B_{r}(0)\times\{k\}} \notin \mathcal{U}_k^{r,+}(0) \}$ is small is completely analogous, using $\overline\lambda_k(y) := \argsup_{\lambda \in [\zeta_k^{r,-}(y),\zeta_k^{r,+}(y)]} \varphi_\mu(\lambda)$ instead of $\underline\lambda_k(y)$ and
 \begin{equation*}
    V_{r}^{-d}\sum_{y \in B_{r}(x)} \varphi_\mu(\overline\lambda_k(y)) \le (1-\delta)\zeta_{k+1}^{r,+}(x)
 \end{equation*}
 instead of \eqref{eqn:concentration_1}.
\end{proof}

\begin{remark}
Lemma~\ref{lem:U_k} is a (slightly) strengthened version of Lemma~2.3, from \cite{BARW}, where it was shown that for any given c.d.p.'s $\zeta^{R,\pm}_k$ and any (possibly random) configuration $\widetilde \eta \in \{0,1\}^{\Z^d}$, such that for some $k \in \{0,\dots,k_0-1\}$ and $x \in \Z^d$
\begin{equation*}
    \delta_R(y;\widetilde \eta) \in \big[\zeta_k^-(y),\zeta_k^+(y) \big],\quad \text{for all } y \in B_R(x),
\end{equation*}
it holds that
\begin{equation*}
    \P\Big( \zeta_{k+1}^-(x) \le \delta_R\big(x;\Phi_{0,1}(\widetilde \eta)\big) \le \zeta_{k+1}(x) \Big| \mathcal{F} \Big)\ge 1-2e^{-cV_R^d}
\end{equation*}
for some $c>0$ where $\mathcal{F} = \sigma(\widetilde \eta(x) : x \in \Z^d )\vee \sigma(U(x,0) : x \in \Z^d)$. 
\end{remark}
The uniform (in configurations of $\eta$) concentration result of local $r$-densities of Lemma~\ref{lem:U_k} plays an important role in the following section, where it is used as a building block for the set $G_U$ of good driving noise.

Our proofs of Propositions~\ref{prop:main_proposition} and \ref{prop:speed_bound} are based on two specific choices $\zeta^{\pm,r_0}$, $\zeta^{R,\pm}$ of c.d.p.'s which are slight modifications of the c.d.p.'s introduced in \cite{BARW}. They have the same general shape as sketched in Figure \ref{fig:zeta_sketch} (for $d=1$), but with scale-dependent fronts (i.e.\ the steps in the ``staircase'' that make up the fronts have different lengths for $r=r_0$ and for $r=R$). We describe these in more detail now.

The construction of the c.d.p.'s we work with relies on a family of auxiliary functions, which are constructed explicitly in Lemma~2.5, Lemma~2.6 of \cite{BARW} and the discussion following those results. To introduce these auxiliary functions, let $r\in \N$ and recall the definition of $(\alpha_m)_{m\ge 1}$ from Lemma~\ref{lem:varphi_properties} and the definition of $m_0$ from \eqref{def:m_0}.
With these  objects at hand, we consider for any integer $R_\maximal\in \N$ and $\varepsilon_0 \in (0,\alpha_1)$, $s \in (0,1)$ and $w\ge 2$ the family $(\chi_k^r(x))_{k\ge 0}$ of functions defined as follows. On $\{ \norm{x} \le R_\maximal + m_0r+ k\lceil sr \rceil\}$, we let $\chi_k^r(x) \equiv \alpha_{1}$, and for all $x$ outside of this set, we let
\begin{equation}
    \label{eqn:auxillary_family}
    \chi_k^r(x) =
                    \alpha_{1}\prod_{i=1}^d\min\Big\{\Big( \big(\varepsilon_0/\alpha_1)^{1/d}+ \frac{\widetilde R_k - |x_i|}{\lceil wr \rceil}\Big)\ind_{\{\widetilde R_k \ge |x_i|\}} , 1 \Big\},
\end{equation}
where $\widetilde R_k = R_\maximal +m_0r+ k\lceil sr\rceil + \lceil rR \rceil$, $x_i$ is the $i$-th coordinate of $x$. Note that for $d=1$, the non-constant section of $\chi_k^r$ has a width of $\lceil wr\rceil$, and on this section, the function decreases linearly to the value $\varepsilon_0$. Moreover, in this case, $\chi_{k+1}^r(|x|) = \chi_k^r(|x|-\lceil sr \rceil)$, i.e.\ in the one-dimensional case, increasing $k$ by one shifts the non-constant parts of $\chi_k^r$ to the outside by $\lceil sr\rceil$. Note that these properties also hold along the coordinate axis when $d>1$.

For convenience, we set $R_\maximal^r(k) := R_\maximal + k\lceil sr\rceil $ for $k \in \{0,\dots,k_0\}$ where $k_0$ is any finite value (once the scales are defined properly, we will take $k_0 = L_t$). Note that with this notation, $\chi_k^r\equiv\alpha_1$ on the ball of radius $R_\maximal^r(k)+m_0r$. 
With the help of the family $(\chi_n^r)_{n\ge0}$, we introduce the following functions
\begin{equation}
    \label{eqn:new_zeta_minus}
        \zeta_{k}^{r,-}(x) =\begin{cases}
                            \alpha_{m_0} \quad &\text{if $\norm{x}\le R_\maximal^r(k)$}, \\
                            \alpha_{m_0-j+1} &\text{if $R_\maximal^r(k)+(j-1)r\le\norm{x}\le R_\maximal^r(k)+jr$, $1\le j\le m_0$},\\
                            \chi_k^r(x) &\text{if $\norm{x}\ge R_\maximal^r(k)+m_0r$},
                        \end{cases}
\end{equation}
and
\begin{equation}
    \label{eqn:new_zeta_plus}
        \zeta_{k}^{r,+}(x) =\begin{cases}
                            \beta_{m_0} \quad &\text{if $\norm{x}\le R_\maximal^r(k)$}, \\
                            \beta_{m_0-j+1} &\text{if $R_\maximal^r(k)+(j-1)r\le\norm{x}\le R_\maximal^r(k)+jr$, $1\le j\le m_0$},\\
                            1\vee \beta_1 &\text{if $\norm{x}\ge R_\maximal^r(k)+m_0r$}.
                        \end{cases}
\end{equation}
For the sake of readability, we do not make the dependence of these functions on $R_\maximal$, $(\alpha_{m})_{m\ge0}$, $s$, $w$, $\varepsilon_0$ explicit in the notation. Moreover, $\zeta_k^{r,-}$ is supported on a ball of radius $R_\maximal^r(k)+m_0r+\lceil wr \rceil$ and gives the strongest density control on a ball of radius $R_\maximal^r(k)$. The next result shows that these families of functions are c.d.p.'s for large enough choices of $r$.

\begin{lemma}\label{lem:zeta_is_cdp}
    For any $M\ge 1$ there exists $R_{\mu,M}$ such that for all $R\ge R_{\mu,M}$ and $r \ge \lceil R/M \rceil$ there exists $s \in (0,1)$, $w \ge 2$ and $\varepsilon_0,\delta_0>0$ such that for $R_\maximal \ge 2R$ the family of functions $\zeta^{r,\pm}$ as defined in \eqref{eqn:auxillary_family}--\eqref{eqn:new_zeta_plus} are $(\varepsilon_0,\delta_0,r)$-c.d.p.'s in the sense of Definition~\ref{def:density_profiles}.
\end{lemma}

\begin{proof}
    It is a direct consequence of Lemma~2.5 in \cite{BARW} that for large enough $r$ there exists $s \in (0,1), w\ge 2$ and  $\varepsilon_0\in (0,\alpha_1)$ of Lemma~4.2 in \cite{BARW} (literally a trivial modification of this result, as $R_\maximal$ may differ from the corresponding quantity there) that there exists $\delta_0>0$ such that the functions $\zeta^{r,\pm}$ are $(\varepsilon_0,\delta_0,r)$-c.d.p.'s. 
    For fixed $M\ge 1$, the lemma follows by taking $R_{\mu,M}$
    to be the smallest value of $R$ such that Lemma~2.5 and 4.2 of \cite{BARW} are applicable with $\lceil R/M \rceil$.
\end{proof}

Without loss of generality, we can take the same value of $\varepsilon_0,\delta_0>0$ for the two c.d.p.'s $\zeta^{\pm,r_0}$ and $\zeta^{R,\pm}$, where $r_0$ is the proportion of $R$ that is fixed in Section~\ref{sec:proof_speed_prop}. By taking $R$ large enough, Lemma~\ref{lem:zeta_is_cdp} guarantees that both $\zeta^{\pm,r_0}$ and $\zeta^{R,\pm}$ are c.d.p.'s in the sense of Definition~\ref{def:density_profiles}.
Note also that for $r=R$ and up to the length $2R_\maximal$ (i.e.\ of the constant centre section of 
\eqref{eqn:new_zeta_minus}--\eqref{eqn:new_zeta_plus}) these functions correspond exactly to the profiles used in \cite[Section~4]{BARW}.

\section{Proof of Proposition~\ref{prop:main_proposition}}
\label{sec:proof_main}

We now have all the tools in order to prove Proposition~\ref{prop:main_proposition}. We introduce scales $L_s$ and $L_t$ and sets $G_\conf$, $G_U$, in dependence of $R$ and show that with increasing $R$, the properties \ref{prop:item1}--\ref{prop:item2} hold with arbitrarily high probability.

We define the spatial and temporal scales $L_s,L_t$ as follows. First, we introduce $R_\maximal := c_\dens \lceil R\log R \rceil$, where $c_\dens :=1+2c_\ttime$ and $c_\ttime> -(d+1)/\kappa(\mu,\varepsilon_\FP)$ and then choose 
\begin{equation}
    \label{eqn:new_scale}
    L_s  := R_\maximal \quad
  \text{and} \quad L_t := T^\spread + T^\couple,
\end{equation}
where 
\begin{equation}
    \label{eqn:T_spread_and_T_couple}
    T^\spread :=   \big\lceil 3 c_\dens \lceil R \log R \rceil / \lceil sR \rceil \big\rceil,\quad \text{and}\quad T^\couple := c_\ttime \lceil \log R \rceil.
\end{equation}
With this choice of scales, it follows immediately that the support of $\zeta_0^{R,-}$ is contained in $B_{2L_s}(0)$ since ${((m_0+1)R + \lceil wR \rceil)<L_s}$. 

We use the c.d.p.'s $\zeta^{\pm,r_0}$ and $\zeta^{R,\pm}$, cf. \eqref{eqn:new_zeta_minus}--\eqref{eqn:new_zeta_plus}, to define the set $G_\conf$ of good local configurations as follows

\begin{equation}
 \label{eqn:def_G_conf}
   G_\conf := \left\{ \widetilde \eta \in  \{0,1\}^{B_{2L_s}(0)} :  \;
      \parbox{0.638\textwidth}{$
       \zeta_0^{R,-}(y) \le \delta_{R}(y;\widetilde \eta) \le \zeta_0^{R,+}(y)\,\text{ for }\, y \in \Supp( \zeta_0^{R,-}), \text{ and }\\
        \zeta_0^{-,r_0}(y) \le \delta_{r_0}(y;\widetilde \eta) \le \zeta_0^{+,r_0}(y)\,\text{ for }\, y \in \Supp\big( \zeta_0^{-,r_0}\big)$} \: 
      \right\}.
\end{equation}
The set $G_\conf$ should be seen as the property which replaces the \emph{well-startedness} property of \cite{BARW}, cf. \eqref{eqn:well-started2}. We note again that the control on the local $r_0$-densities that we ask for here is a technicality that is needed in Section~\ref{sec:proof_speed_prop} and has no analogue in the renormalisation construction of \cite{BARW}.

For the sake of notational convenience, we also introduce the ``cylinder set'' of $G_\conf$ defined by the configurations that are locally in $G_\conf$,
\begin{equation}
    \label{eqn:def_G_conf_hat}
    \widehat G_\conf := \left\{ \widetilde \eta \in \{0,1\}^{\Z^d} : \widetilde \eta|_{B_{2L_s(0)}} \in G_\conf \right\}.
\end{equation}

The introduction of the set $G_U$ of good driving noise is a bit more subtle. In order for properties~\ref{prop:item1}--\ref{prop:item2} to be satisfied, we require that for any $(x,n) \in \Z^d\times \Z$ such that $\Gamma(x,n) = 1$, cf. \eqref{eqn:oriented_percolation}, the driving noise on $\bblock_{4}(x,n)$ is such that it ensures the following two items:
\begin{enumerate}[label =(\Alph*)]
    \item  Whenever ${\eta_{nL_t}|_{B_{2L_s}{(L_sx)}} \in G_\conf}$ the strongest control of local $R$-densities by the c.d.p.'s $\zeta_k^{R,\pm}$ from \eqref{eqn:new_zeta_minus}--\eqref{eqn:new_zeta_plus} spreads throughout the entire spatial extent of $\bblock_{4}(x,n)$ by time $T^\spread$, and the expanding control of the local $r_0$-densities holds throughout the block. I.e.\
    \begin{align*}
        \alpha_{m_0}\le \delta_R(z ; \eta_{nL_T+T^\spread})\le \beta_{m_0}, \quad & \text{for all }\, z\in B_{4L_s}(L_sx)\\
        \zeta_k^{-,r_0}(z) \le \delta_{r_0}(z;\eta_k)\le \zeta_k^{+,r_0}(z), \quad & \text{for all }\, z\in B_{4L_s}(L_sx), k \in \{0,\dots, L_t\}
    \end{align*}
    
    \label{list:A}
    
    \item Given that the control of $R$-densities has spread as in the first item, the process couples successfully to any reference configuration on $B_{3L_s}(L_sx)$ in an additional $T^\couple$ time steps, where the reference configurations are given by
    \begin{equation}
    \label{eqn:reference_configurations}
        \mathrm{C}_\reff = \left \{ \eta^\reff \in \{0,1\}^{\Z^d} :  | \delta_R(\cdot;\eta^\reff)- \theta_\mu | < \varepsilon_\FP \right \}\subseteq \widehat G_\conf.
\end{equation}
See also  Remark~\ref{rem:new_timescale} above. \label{list:B}
\end{enumerate}
Let us denote the event of \ref{list:A} occurring by $A^\spread(x,n)$ and the event of \ref{list:B} occurring by $A^\couple(x,n)$. The set $G_U$ can then be defined implicitly as the set of driving noise configurations such that
\begin{equation}
    \big\{ U|_{\bblock_{4}(0,0)} \in G_U \big\} = A^\spread(0,0)\cap A^\couple(0,0).
\end{equation}

Note that $T^\spread$ was chosen to be the time that it takes for the c.d.p.'s $\zeta_k^{R,\pm}$ to spread so far that $\zeta_{T^\spread}^{R,-}|_{B_{4L_s}(0)}\equiv \alpha_{m_0}$. The probability of $A^\spread(x,n)$ can easily be bound with the help of Lemma~\ref{lem:U_k}.

\begin{lemma}\label{lem:A_spread}
    It holds that
    \begin{equation*}
    \P\big(A^\spread(x,n)\big) \ge 1-q^{(1)}(R,\mu),\quad (x,n) \in \Z^d\times \Z, 
    \end{equation*}
    where $q^{(1)}(R,\mu) \downarrow 0$ for $R \to \infty$.
\end{lemma}

\begin{proof}
    Without loss of generality, we consider only the case $(x,n) = (0,0)$. The proof follows directly by applying a union bound and using Lemma~\ref{lem:U_k}. Indeed,
    if we denote by $\mathcal{C}_R$ the (spatial) $R$-fattening of the (truncated) cone
     $\bigcup_{m=1}^{T^\spread} \Supp(\zeta_m^{R,-})\times \{m\}$  and by $\mathcal{C}_{r_0}$ the (spatial) $r_0$-fattening of the (truncated) cone  $\bigcup_{m=1}^{L_t} \Supp(\zeta_m^{-,r_0})\times\{m\}$ then Lemma~\ref{lem:U_k} (applied once with $r=R$ and once with $r=r_0$) gives
    \begin{align*} 
    1-\P\big(A^\spread\big(0,0)) &\le \P\big(\exists (y,k) \in \mathcal{C}_R  : U|_{B_{R}(y) \times \{k\}} \notin \mathcal{U}_{k}^{R}(y)\big)\\
    & \quad + \P\big(\exists (y,k) \in \mathcal{C}_{r_0}  : U|_{B_{r_0}(y) \times \{k\}} \notin \mathcal{U}_{k}^{r_0}(y) \big)\\
    &\le c T^\spread \textrm{Volume} \big(\mathcal{C}_R\big)\exp(- c' V_R^d) \\
    & \quad + c L_t \textrm{Volume}\big(\mathcal{C}_{r_0}\big) \exp(-c' V_{r_0}^d),
    \end{align*}            
    for some constants $c,c'>0$. As $T^\spread$ and $L_t$ are of order $\log R$ and $\textrm{Volume} (\mathcal{C}_R)$, $\textrm{Volume} (\mathcal{C}_{r_0})$ are polynomial in $R$ (recall $r_0$ is a fixed proportion of $R$), the right hand side of the above display tends to zero as $R$ tends to infinity. 
\end{proof}

The next lemma shows that we have a corresponding bound for the probability of the event $A^\couple(x,n)$.

\begin{lemma}\label{lem:A_couple}
    For large enough $R$, it holds that
    \begin{equation*}
        \P\big( A^\couple(x,n)\big)\ge 1-q^{(2)}(R,\mu),
    \end{equation*}
    where $q^{(2)}(R,\mu)\downarrow 0$ for $R \to \infty$.
\end{lemma}

\begin{proof}
Without loss of generality, we set again $(x,n) = (0,0)$. The proof is an adaptation of arguments that can already be found in \cite[Lemma~4.6]{BARW}.

We start with the trivial observation that $A^\couple(0,0)$ gives conditions on the behaviour of particle configurations $\widetilde \eta \in \{0,1\}^{\Z^d}$ which satisfy the local density condition
\begin{equation}
    \label{eqn:density_condtion_A}
 \zeta_{T^\spread}^{R,-}(y)<\delta_R(y; \widetilde \eta) < \zeta_{T^\spread}^{R,+}(y), \quad \text{for all } y \in \Supp(\zeta_{T^\spread}^{R,-}).
\end{equation}
For later convenience, we denote the set of all such particle configurations by $\widehat G_\conf^{(2)}$. Moreover, we use the following convention throughout the proof: $\widetilde \eta_n := \Phi_{0,n}(\widetilde \eta)$ for any $\widetilde \eta \in \{0,1\}^{\Z^d}$.
With this notation, we can write
\begin{equation*}
      A^\couple(0,0)^c = \Big\{ 
      \exists y \in B_{3L_s}(0), \exists \widetilde \eta \in \widehat G_\conf^{(2)}, \exists\eta^\reff \in \mathrm{C}_\reff :  \widetilde \eta_{T^\couple}(y) \neq \eta^\reff_{T^\couple}(y)
      \Big\}
\end{equation*}

The proof of the lemma is based on the following calculation. Recall the definition of the $\sigma$-algebras from \eqref{eqn:sigma_algebra}. For any $y \in \Z^d$ and $k > 1$, cf. \eqref{eqn:sigma_algebra}, it holds by Markov's inequality that
\begin{align}
\begin{split}
    \label{eqn:iteration1}
       \P&\Big(\exists  \eta^{(1)}, \eta^{(2)} \in \widehat G_\conf^{(2)} : \eta^{(1)}_k(y) \neq \eta^{(2)}_k(y) \Big| \mathcal{G}_{0,k} \Big) \le
   \E\Big[ \sup_{\widetilde\eta \in \widehat G_\conf^{(2)}}\widetilde\eta_k(y)-\inf_{\widetilde \eta \in \widehat G_\conf^{(2)}}\widetilde\eta_k(y)\Big |  \mathcal{G}_{0,k}\Big]\\
   &= \E\Big[ \sup_{\widetilde \eta \in \widehat G_\conf^{(2)}} \ind_{\{U(y,k)\le \varphi_\mu(\delta_R(y;\widetilde\eta_{k-1}))\}}-\inf_{\widetilde \eta \in 
 \widehat G_\conf^{(2)}}\ind_{\{U(y,k)\le \varphi_\mu(\delta_R(y;\widetilde\eta_{k-1}))\}} \Big|  \mathcal{G}_{0,k} \Big]\\
   &= \E\Big[ \ind_{\{U(y,k)\le \sup_{\widetilde \eta \in \widehat G_\conf^{(2)}} \varphi_\mu(\delta_R(y;\widetilde\eta_{k-1}))\}}-\ind_{\{U(y,k)\le \inf_{\widetilde \eta \in \widehat G_\conf^{(2)}}\varphi_\mu(\delta_R(y;\widetilde\eta_{k-1}))\}} \Big|  \mathcal{G}_{0,k} \Big]\\ 
   &= \P\Big( U(y,k) \le \big|\sup_{\widetilde \eta \in \widehat G_\conf^{(2)}} \varphi_\mu\big(\delta_R(y;\widetilde\eta_{k-1})\big) - \inf_{\widetilde \eta \in \widehat G_\conf^{(2)}}\varphi_\mu\big(\delta_R(y;\widetilde\eta_{k-1})\big)\big| \Big |  \mathcal{G}_{0,k}\Big)\\
   &=\Big|\sup_{\widetilde \eta \in \widehat G_\conf^{(2)}} \varphi_\mu\big(\delta_R(y;\widetilde\eta_{k-1})\big) - \inf_{\widetilde\eta \in \widehat G_\conf^{(2)}}\varphi_\mu\big(\delta_R(y;\widetilde\eta_{k-1})\big)\Big|
   \end{split}
\end{align}
The second equality holds because the supremum and infinmum are really a maximum and minimum over what happens in the $R$-neighbourhood of $y$, which only involves finitely many local configurations.

Now, if we had uniformly in the configurations in $\widehat G_\conf^{(2)}$ that $|\delta_R(y;\Phi_{0,k-1}(\cdot))-\theta_\mu|<\varepsilon_\FP$, i.e.\ the local density around $x$, at time $k-1$ were uniformly in the region where $\varphi_\mu$ is a contraction, cf.~Lemma~\ref{lem:varphi_properties}(c), then we would get
\begin{align}
\begin{split}
\label{eqn:iteration2}
   \Big|\sup_{\widetilde\eta \in \widehat G_\conf^{(2)}} &\varphi_\mu\big(\delta_R(y;\widetilde\eta_{k-1})\big) - \inf_{\widetilde \eta \in \widehat G_\conf^{(2)}}\varphi_\mu\big(\delta_R(y;\widetilde\eta_{k-1})\big)\Big|\\
&\le 
\kappa(\mu,\varepsilon_\FP) \Big| \sup_{\widetilde \eta \in \widehat G_\conf^{(2)}} \delta_R(y;\widetilde\eta_{k-1})- \inf_{\widetilde \eta \in \widehat G_\conf^{(2)}}\delta_R(y;\widetilde\eta_{k-1}) \Big|\\
&\le \kappa(\mu,\varepsilon_\FP) V_R^{-d}\sum_{y_1 \in B_R(y)} \Big|\sup_{\widetilde\eta \in \widehat G_\conf^{(2)}}\widetilde \eta_{k-1}(y_1)-\inf_{\widetilde \eta \in \widehat G_\conf^{(2)}}\widetilde \eta_{k-1}(y_1)\Big|,
\end{split}
\end{align}
and we could (contingent on having corresponding density control in a slightly larger region, i.e.\ an $R$-fattening of the region, in order to apply Lemma~\ref{lem:varphi_properties}(c) again) iterate this calculation.

In order to formalise the uniform density control that lets us apply Lemma~\ref{lem:varphi_properties}(c) in the above, we introduce for any $r \in \N$ functions $\psi_r: \{0,1\}^{\Z^d}\to \{0,1\}$ with
\begin{equation*}
    \psi_r(\widetilde \eta) = \ind_{\{ \delta_R(z;\widetilde \eta) \in [\theta_\mu-\varepsilon_\FP, \theta_\mu + \varepsilon_\FP]\, \text{ for all }\, z\in B_r(0)\}},\quad \widetilde \eta \in \{0,1\}^{\Z^d},
\end{equation*}
and the event
\begin{equation}
    \label{eqn:Psi_T_couple}
    \Psi_{T^\couple} :=\Big\{\psi_{3L_s+T^\couple \lceil sR \rceil + (T^\couple-l)R}(\widetilde \eta_{T^\couple-l}) = 1 : \,l = 1,\dots, T^\couple, \widetilde \eta \in \widehat G_\conf^{(2)} \Big\}.
\end{equation}
On this event, we have the necessary density control in order to iterate the calculation in \eqref{eqn:iteration1}--\eqref{eqn:iteration2}.

It is not hard to see that $\Psi_{T^\couple}$ has high probability. Note that for any $\widetilde \eta \in \widehat G_\conf^{(2)}$, it holds that $\psi_{4L_s}(\widetilde \eta) = 1$ and 
 and by definition $T^\couple \lceil sR \rceil + T^\couple R = c_\ttime \lceil sR \rceil \lceil \log R \rceil + c_\ttime \lceil \log R \rceil R \le 2c_\dens \lceil R\log R \rceil = L_s$, so that the case $l= T^\couple$ in \eqref{eqn:Psi_T_couple} is satisfied.
Then, using the same argument as in the proof of Lemma~\ref{lem:A_spread}, this density control by the $\zeta^{R,-}_k$ profiles spreads by $\lceil sR\rceil $ in every time step, with high probability, such that
\begin{equation*}
    \P(\Psi_{T^\couple})\ge 1-\widetilde q^{(1)}(R,\mu),
\end{equation*}
for some $\widetilde q^{(1)}(R,\mu)>0$ such that $\lim_{R \to \infty}\widetilde q^{(1)}(R,\mu) = 0$.

With this, we can now prove the lemma. We consider the probability of the complement of $A^\couple(0,0)$ conditioned on $\Psi_{T^\couple}$,
\begin{align*}
     \P\Big(& A^\couple(0,0)^c \Big | \Psi_{T^\couple}\Big) \\
     = &\P \Big( \exists y \in B_{3L_s}(0), \exists \widetilde \eta \in \widehat G_\conf^{(2)}, \exists\eta^\reff \in \mathrm{C}_\reff :  \widetilde \eta_{T^\couple}(y) \neq \eta^\reff_{T^\couple}(y) \Big | \Psi_{T^\couple}\Big).
\end{align*}
Moreover, since $\mathrm{C}_\reff \subseteq \widehat G_\conf^{(2)}$, cf.~\eqref{eqn:reference_configurations} and \eqref{eqn:density_condtion_A}, and writing $T=T^\couple$ in the following, the right hand side of the last display is bounded from above by 
\begin{equation}
     \P\Big( \max_{y \in B_{3L_s}(0)}\big( \sup_{\widetilde \eta \in \widehat G_\conf^{(2)}}\psi_{L_s}(\widetilde \eta_{T})\widetilde \eta_{T}(y)- \inf_{\widetilde \eta \in \widehat G_\conf^{(2)}}\psi_{L_s}(\widetilde \eta_{T})\widetilde \eta_{T}(y)\big) \ge 1 \Big | \Psi_{T} \Big ).
\end{equation}

This probability can now be dealt with as in the iteration of \eqref{eqn:iteration1}--\eqref{eqn:iteration2}. Indeed, for any $k \in \{1,\dots, T\}$, one gets with a union bound and Markov's inequality that

\begin{align*}
    &\E\Big[\P\Big( \max_{y \in B_{3L_s}(0)}\big( \sup_{\widetilde \eta \in \widehat G_\conf^{(2)}}\psi_{L_s}(\widetilde \eta_{k})\widetilde \eta_{k}(y)- \inf_{\widetilde \eta \in \widehat G_\conf^{(2)}}\psi_{L_s}(\widetilde \eta_{k})\widetilde \eta_{k}(y)\big) \ge 1\Big |\mathcal{G}_{0,k}\Big)\Big | \Psi_{T}\Big]\\
    &\le\Big[ \sum_{y \in B_{3L_s}(0)}\E\Big[ \sup_{\widetilde \eta \in \widehat G_\conf^{(2)}}\psi_{L_s}(\widetilde \eta_{k})\widetilde \eta_{k}(y) - \inf_{\widetilde \eta \in \widehat G_\conf^{(2)}}\psi_{L_s}(\widetilde \eta_{k})\widetilde \eta_{k}(y)\big) \Big| \mathcal{G}_{0,k} \Big] \Big |\Psi_{T} \Big] \\
    &\le \kappa(\mu,\varepsilon_\FP) V_R^{-d} \Big[ \sum_{y \in B_{3L_s}(0)} \sum_{y_1 \in B_R(y)}\E\big[ \sup_{\widetilde \eta \in \widehat G_\conf^{(2)}} \widetilde \eta_{k-1}(y_1) - \inf_{\widetilde \eta \in \widehat G_\conf^{(2)}} \widetilde \eta_{k-1}(y_1)\big| \mathcal{G}_{0,k}\big] \Big |\Psi_{T}\Big]\\
    &\le 
    \kappa(\mu,\varepsilon_\FP)^{k} V_R^{-dk} \E\Big[ \sum_{y \in B_{3L_s}(0)} \sum_{y_1 \in B_R(y)}\sum_{y_2 \in B_R(y_1)}\cdots\\
    & \hspace{4cm} 
    \sum_{y_{k} \in B_R(y_{k-1})} \Big(\sup_{ \widetilde \eta \in \widehat G_\conf^{(2)}} \widetilde \eta(y_k)- \inf_{\widetilde \eta \in \widehat G_\conf^{(2)}} \widetilde \eta(y_k) \Big) \Big | \Psi_{T} \Big] \\
    &\le \kappa(\mu,\varepsilon_\FP)^kV_{3L_s}^d.
\end{align*}

By the choice of $c_\ttime$, see above \eqref{eqn:new_scale}, the last line tends to zero as $R$ tends to infinity.
The claim follows, as 
\begin{align*}
        \P\big( A^\couple(0,0)\big) &=\P\big( A^\couple(0,0)\big | \Psi_{T^\couple} \big) 
        \times\P\big(\Psi_{T^\couple} \big)\\
        &\ge \big(1- \kappa(\mu,R)^{T^\couple}V_{3L_s}^d\big)\big(1-\widetilde q^{(1)}(R,\mu)\big),
\end{align*}

the right hand of which can be made to be arbitrarily close to one, by choosing $R$ large.
\end{proof}

By Lemmas~\ref{lem:A_spread} and \ref{lem:A_couple} the proof of Proposition~\ref{prop:main_proposition} follows directly with 
\begin{equation*}
\widetilde R_{\mu,\varepsilon} := \inf\{R>0 : \max\{q^{(1)}(R,\mu),q^{(2)}(R,\mu)\}<\varepsilon\}.
\end{equation*}

\section{Proof of Proposition~\ref{prop:speed_bound}}\label{sec:proof_speed_prop}
Let $(x,n)\in \Z^d \times \Z$ be such that $\Gamma(x,n) = 1$. In order to prove Proposition~\ref{prop:speed_bound}, we require uniform in $z \in B_{L_s/2}(L_sx)$ control of the conditional probabilities
\begin{equation}
    \label{eqn:closeness_RW_eq1}
    P_\eta\Big( \max_{(n-1)L_t<k\le nL_t} \norm{X_k -z} \ge L_s/4 \Big| X_{(n-1)L_t} = z , \Gamma(x,n) = 1\Big).
\end{equation}
Without loss of generality, we assume that $x=0$ and $n =- 1$, such that $z \in B_{L_s/2}(0)$ and set for readability $P_\eta^z(\cdot):= P_\eta(\cdot |X_{0} = z , \Gamma(0,-1) = 1 )$. In order to upper bound \eqref{eqn:closeness_RW_eq1} we split the random walk increments into a martingale and non-martingale part, i.e.\ 
\begin{equation}
    \label{eqn:def_martinglale_part}
    X_k -X_{k-1} = E_\eta^z\big[X_k-X_{k-1}\big| X_{k-1}\big]+ Y_k,
\end{equation}
where $Y_k$ is the martingale part. We can thus write
\begin{equation}
    \label{eqn:split_martingale}
    X_k -z = \sum_{i=1}^k X_i-X_{i-1} = \sum_{i = 1}^k E_\eta^z\big[ X_i-X_{i-1} \big| X_i\big] + \sum_{i=1}^k Y_i.
\end{equation}
Let us introduce the event 
\begin{equation}
    \label{eqn:A_martingale}
    A_\mart := \Big\{ \max_{0<k\le L_t} \Norm{\sum_{i=1}^k Y_i} \ge L_s/8 \Big\}.
\end{equation}
Using \eqref{eqn:split_martingale} and \eqref{eqn:A_martingale}, we can thus bound \eqref{eqn:closeness_RW_eq1} as follows
\begin{equation}
    \label{eqn:split_RW_1}
    \begin{aligned}
        P_\eta^z&\big( \max_{0<k\le L_t} \norm{X_k-z} \ge L_s/4\big) \\ &\le P_\eta^z\Big(\max_{0<k\le L_t} \Norm{ \sum_{i = 1}^k E_\eta^z[ X_i-X_{i-1} | X_i]}+\max_{0<k\le L_t} \Norm{\sum_{i=1}^k Y_i}\ge L_s/4 \Big) \\
        &\le P_\eta^z\Big(\max_{0<k\le L_t}  \sum_{i = 1}^k \norm{E_\eta^z[ X_i-X_{i-1} | X_i]}\ge L_s/8 ,A_\mart^c \Big)
        + P_\eta^z(A_\mart)
        \end{aligned}
\end{equation}

To deal with the first summand in the last line of \eqref{eqn:split_RW_1}, we first claim that  on $A_\mart^c$ it holds for all $k \in \{0,\dots,L_t\}$ that $\norm{X_k-z} \le L_s/2$.
To see this, note first that on $A_\mart^c$ we can write for any $k\in\{1,\dots,L_t\}$ 
\begin{equation}
    \label{eqn:split_RW_2}
    \norm{X_k-z} \le L_s/8 + \sum_{i=1}^k \norm{E_\eta^z[X_i-X_{i-1} | X_{i-1}]}.
\end{equation}
We can bound the sum of expected differences using the following lemma, the proof of which we postpone to the end of this section.
\begin{lemma}
    \label{lem:non_martingale_part}
    For $k \in \{1,\dots, L_t\}$, it holds on $\{\norm{X_{k-1}-z} \le L_s/2\}$ that
    \begin{equation}
      \norm{E_\eta^z[X_k-X_{k-1} | X_{k-1}]} < \frac{L_s}{8L_t}.
    \end{equation}
\end{lemma}
Now since $X_0=z$ we can apply Lemma~\ref{lem:non_martingale_part} and it follows with \eqref{eqn:split_RW_2} that $\norm{X_1-z} \le L_s/8 + \tfrac{L_s}{8L_t} < L_s/8+L_s/8 = L_s/4$. Thus, applying  Lemma~\ref{lem:non_martingale_part} and \eqref{eqn:split_RW_2} inductively yields
\begin{equation*}
    \norm{X_k-z} \le L_s/8+ \frac{kL_s}{8L_t} \le L_s/4, \quad k \in \{1,\dots,L_t\}.
\end{equation*}
In particular, this implies that on $A_\mart^c$ that
\begin{equation*}
    \max_{0<k\le L_t}  \sum_{i = 1}^k \norm{E_\eta^z[ X_i-X_{i-1} | X_i]} < L_s/8,
\end{equation*}
which in turn implies that the first sum in the last line of \eqref{eqn:split_RW_1} vanishes.

The probability of $A_\mart$ can be dealt with by applying the Azuma-Hoeffding inequality to the partial sum process of $(Y_k)_{k\ge1}$. More precisely, if $S_k := \sum_{i= 1}^{k} Y_i$ and $S_0 = 0$, then it follows by definition, cf.\ \eqref{eqn:def_martinglale_part}, that $\norm{S_k-S_{k-1}} = \norm{Y_k}\le R$ and thus
\begin{equation*}
    P_\eta^z\big( \norm{S_k-S_0}\ge L_s/8\big) \le \exp\Big(\frac{-(L_s/8)^2}{2L_t R^2} \Big), \quad k \in \{0,\dots,L_t\}.
\end{equation*}
By a union bound, it follows that
\begin{equation*}
  P_\eta^z( A_\mart) \le L_t \exp\Big(\frac{-(L_s/8)^2}{2L_t R^2} \Big).
\end{equation*}
The right hand side of the last display is, by the choice of $L_t$ and $L_s$ in \eqref{eqn:new_scale} of order $O(\log(R)/R)$, which tends to zero for $R$ large. We thus set $R_{\mu,\delta,\varepsilon}$ to be the smallest $R$ such that Proposition~\ref{prop:main_proposition} holds and that the above display is smaller than a given $\delta>0$ as in the statement of Proposition~\ref{prop:speed_bound}, i.e.\ for given $\varepsilon,\delta>0$ we set
\begin{equation*}
 R_{\mu,\delta,\varepsilon} := \inf\big\{R\in \N : R> \widetilde R_{\mu,\varepsilon}, L_t \exp\big(-(L_s/8)^2/(2L_tR^2)\big)<\delta\big\}.
\end{equation*}

It still remains to prove Lemma~\ref{lem:non_martingale_part} in order to complete the proof of Proposition~\ref{prop:speed_bound}.

\begin{proof}[Proof of Lemma~\ref{lem:non_martingale_part}]
    By the definition of a good block (i.e.\ $\Gamma(0,-1) = 1$), the driving noise $U|_{\bblock_4(0,-1)}$ is such that the following local $R$- and $r_0$-density conditions are satisfied for $k \in \{0,\dots,L_t\}$,
    \begin{equation}
    \label{eqn:good_density_control_RW}
    \begin{aligned}
        &\zeta_k^{R,-}(x) \le \delta_R(x; \eta_{-L_t+k})\le \zeta_k^{R,+}, \quad \text{for all }\, x \in \Supp(\zeta_k^{R,-}),\\
        &\zeta_k^{-,r_0}(x) \le \delta_{r_0}(x; \eta_{-L_t+k})\le \zeta_k^{+,r_0}, \quad \text{for all }\, x \in \Supp(\zeta_k^{-,r_0}),
        \end{aligned}
    \end{equation}
and in particular, by the definition of the $\zeta^{R,\pm}_k$,
$\zeta^{\pm,r_0}_k$ profiles, cf.\ \eqref{eqn:new_zeta_minus}--\eqref{eqn:new_zeta_plus}, it holds that 
\begin{equation}
    \label{eqn:good_density_control}
    \delta_R(x;\eta_{k}),\delta_{r_0}(x;\eta_k) \in [\theta_\mu-\varepsilon_\FP,\theta_\mu + \varepsilon_\FP],\quad (x,k) \in B_{L_s}(0)\times \{-L_t,\dots,0\}
\end{equation}

Recall that we assumed without loss of generality that  $r_0$ divides $R$. Let $M>1$ denote the resulting quotient, i.e.\ let $M>1$ be such that $R/r_0 = M$.  Then the ball $B_R(0)$ can be divided into $M^d$ sub-balls of radius $r_0$ centred at points ${y_1,\dots,y_{M^d} \in B_R(0)}$. Now by \eqref{eqn:good_density_control} it follows immediately that for any $m \in \{1,\dots, M^d\}$.
\begin{equation*}
    \begin{aligned}
    P_\eta^z( X_k-X_{k-1} \in B_{r_0}(y_m) | X_{k-1})  &= \frac{\delta_{r_0}( y_m+X_{k-1};\eta_{-k-1}) V_{r_0}^d}{\delta_{R}(X_{k-1};\eta_{-k-1})V_R^d} \\
    &\in \Big[ (1-\widetilde \varepsilon) \frac{1}{M^d} , (1+\widetilde \varepsilon) \frac{1}{M^d}\Big],
    \end{aligned}
\end{equation*}
for some $\widetilde\varepsilon = \widetilde\varepsilon(\varepsilon_\FP,R,M,d)>0$, which decreases to zero for $R \to \infty$.
Moreover, denote on $\{  X_k-X_{k-1} \in B_{r_0}(y_m) \}$ the relative displacement of $X_k-X_{k-1}$ with respect to $y_m$ by $z_m$. In particular, it holds that $\norm{z_m} \le r_0$. Then we have that
\begin{equation}
    \begin{aligned}
    \nnorm{E_\eta^z[X_k-X_{k-1} | X_{k-1}]} 
 &=  \Norm{\sum_{m=1}^{M^d} E_\eta^z\big[(y_m+z_m)\ind_{\{(X_k -X_{k-1}\in B_{r_0}(y_m)\}} | X_{k-1} \big]}\\
 &\le \Norm{\sum_{m=1}^{M^d} E_\eta^z\big[y_m\ind_{\{(X_k -X_{k-1}\in B_{r_0}(y_m)\}} | X_{k-1} \big]} \\ & \quad + \Norm{\sum_{m=1}^{M^d} z_m P_\eta^z(X_k -X_{k-1}\in B_{r_0}(y_m) | X_{k-1} )}\\
 &\le \widetilde \varepsilon +(1+\widetilde \varepsilon)r_0 \le 2r_0.
 \end{aligned}
\end{equation}
 Recall that $L_s = c_\dens \lceil R \log R \rceil$. If we write $c_1$ for the constant such that $L_t= c_1 \log R$, it follows for any fixed $M> 16c_1/c_\dens$  that $2r_0 = 2R/M < L_s/8L_t$.
\end{proof}

\section{Funding} 
  The work was supported by Deutsche Forschungsgemeinschaft through DFG
  project no.\ 443869423 and by Schweizerischer Nationalfonds through SNF
  project no.\ 200021E\_193063 in the context of a joint DFG-SNF
  project within in the framework of DFG priority programme SPP 2265
  Random Geometric Systems.

\bibliographystyle{alpha}
\bibliography{ancestral_BARW_arxiv}

\end{document}